\documentclass[12pt]{article}%
\usepackage{amsmath}
\usepackage{amsfonts}
\usepackage{amssymb}
\usepackage{float}
\usepackage{graphicx}%
\providecommand{\U}[1]{\protect\rule{.1in}{.1in}}
\newtheorem{theorem}{Theorem}

\newtheorem{definition}[theorem]{Definition}

\newtheorem{lemma}[theorem]{Lemma}

\newtheorem{proposition}[theorem]{Proposition}

\newenvironment{proof}[1][Proof]{\noindent\textbf{#1.} }{\ \rule{0.5em}{0.5em}}
\begin{document}

\title{Social Distancing, Gathering, Search Games: \\Mobile Agents on Simple Networks}
\author{Steve Alpern$^{1}$ and Li Zeng$^{2}$\\$^{1}$Warwick Business School, $^{2}$Department of Statistics,\\University of Warwick, Coventry CV2 7AL, United Kingdom}
\date{}
\maketitle

\begin{abstract}
During epidemics, the population is asked to Socially Distance, with pairs of
individuals keeping two meters apart. We model this as a new optimization
problem by considering a team of agents placed on the nodes of a network.
Their common aim is to achieve pairwise graph distances of at least $D,$ a
state we call \textit{socially distanced}. (If $D=1,$ they want to be at
distinct nodes; if $D=2$ they want to be non-adjacent.) We allow only a simple
type of motion called a Lazy Random Walk: with probability $p$ (called the
\textit{laziness} parameter), they remain at their current node next period;
with complementary probability $1-p$ , they move to a random adjacent node.
The team seeks the common value of $p$ which achieves social distance in the
least expected time, which is the absorption time of a Markov chain.

We observe that the same Markov chain, with different goals (absorbing
states), models the gathering, or multi-rendezvous problem (all agents at the
same node). Allowing distinct laziness for two types of agents (searchers and
hider), extends the existing literature on predator-prey search games to
multiple searchers.

We consider only special networks: line, cycle and grid.

Keywords: epidemic, random walk, dispersion, rendezvous search

\end{abstract}

\newpage

\section{Introduction}

To combat epidemics, three actions are recommended to the public: mask
wearing, hand washing, social distancing. This paper models the last of these
in an abstract model of mobile agents on a network. Social distancing can be
considered a group goal (common-interest game) or individual goal
(antagonistic, non-cooperative game). We consider both goals in a dynamic
model where agents (players) walk on a network (graph). A group of $m$
players, or agents, is placed in some way on the nodes of a network $Q.$ Each
agent adopts a Lazy Random Walk (LRW) which stays at his current node with a
probability $p$ (called laziness) and moves to a random adjacent node with
complementary probability $1-p$ (called \textit{speed}). In the common
interest game, we seek a common value of $p$ which minimizes the time for all
pairs of players to be at least $D$ nodes apart (socially distanced). Once $p$
is adopted by all, the positions of the agents on the network (called states)
follow a Markov chain, with distanced states as absorbing. Standard elementary
results on absorption times for Markov chains are used to optimize $p$, to
find the value of $p$ which if adopted by all agents minimizes the absorption
time. This work can be seen as an extension to networks of the \textit{spatial
dispersion} problem introduced by Alpern and Reyniers (2002), where agents
could move freely between \textit{any} two locations. That paper also allowed
agents knowledge of populations at other locations, whereas here agents have
no knowledge of the whereabouts of other agents. (In some cases we do allow an
agent to know the population at his \textit{current node}, with laziness
$p_{i}$ at such a node dependent on the number $i$ of agents at the node.)

We observed that by changing the states which we call absorbing (and not
allowing transitions out of these), we can usefully model other existing
problems. For example, by taking the absorbing states as those where all
agents occupy a common node, we model the multi-rendezvous (or gathering)
problem, and our results extend known results for $m=2$ agents on a network
found in Alpern (1995). We assume the \textit{sticky} form of the problem,
where agents who meet coalesce into a single agent and move together
subsequently. Rendezvous problems were introduced by Anderson and Weber (1990)
and Alpern (1995) in discrete and continuous models. Rendezvous with agents on
graphs have been studied in Alpern, Baston and Essegaier (1999) and Alpern
(2002b). See also Gal (1999), Howard (1999) and Weber (2012). A survey is
given in Alpern (2002).

We also consider a model with two types of agents called \textit{searchers}
and \textit{hiders} (or predators and prey), who can choose different speeds.
Here the payoff is related to the search or capture time (when a searcher
coincides with a hider), with the searchers as minimizers and the hiders as
maximizers in a two-person team search game. These problems were proposed by
Isaacs (1965) and first studied by Zelikin (1972), Alpern (1974) and Gal
(1979), and later by many others. See Gal (1980) and Alpern and Gal (2003) for
monographs on search games. Until now, such problems with capture time payoffs
have mostly had one searcher and one hider. All of these problems start in
some prescribed, or possibly random, position and end when the desired
position is reached. For rendezvous or search (hide-seek) reaching the desired
position clearly ends the game, as all agents will know this. For dispersion,
some binary signal (a siren from an observing drone) might ring until a
distanced position is reached.

As mentioned above, the forerunner to social distancing problems is the
related spatial dispersion problem of Alpern and Reyniers (2002). They
consider $n$ \textit{locations} (with no network structure, so they did not
call them nodes, as we do here) with $n$ agents placed randomly on them. The
aim to obtain a \textit{dispersed} situation with one agent at each location
(two agents are not allowed to be at the same location). This would correspond
to a value of pairwise distance at least $D=1$ and $m=n$ in our model. We will
also consider this situation at times in this paper. They also considered
$m=kn$ agents with the aim of getting $k$ agents at each location. After each
period, the distribution of agents over the locations becomes common
knowledge. That problem modeled a situation where drivers can take any of $n$
bridges, let us say from New Jersey into Manhattan, and the distribution of
yesterday's traffic is announced every night. The aim is to equalize traffic
over the bridges for the common good. Grenager et al (2002) extended that work
to computer science areas and Blume and Franco (2007) to economics. See also
Simanjuntak (2014).

It is clear that ours is an extremely abstract approach to the problem of
social distancing. For a very recent practical analysis of the impact of
social distancing on deaths from Covid-19, including a monetary equivalent,
see Greenstone Nigam (2020). As this is the first paper to introduce the
Social Distancing Problem, we confine ourselves to the consideration of some
simple networks of small size: the line, cycle and lattice (grid) networks.

This paper is organized as follows. Section 2 describes our dynamic model of
agents moving on a network according to a common lazy random walk and derives
the associated Markov chain. A formula for the time to absorption (desired
state) is derived. Section 3 gives some simple examples where the agents
attempt to social distance on the cycle graph $C_{n}$. Section 4 considers
several different problems, all having three agents on $C_{3}:$ social
distancing (4.1), gathering (4.2), a zero sum game with a team of two
searchers seeking one hider (4.3). Section 5 considers gathering (5.1) and
social distancing (5.2) on $C_{5}.$ Section 6 considers a game where $n$
players start together an end of a line graph, each choosing their own
laziness in a lazy random walk. At the first time periods where some agents
are alone at their node, these agents split a unit prize. When $n=2$ (6.1)
\textit{any} pair $\left(  p,p\right)  $ is an equilibrium, but when $n=2$
(6.2) there is no symmetric equilibrium. In Section 7, we use Monte Carlo
simulation to study social distancing on larger grid (7.1) and line (7.2)
graphs. Section 8 concludes.

\section{The Dynamic Model: States and Lazy Random Walks}

The $m$ agents in our model move on a connected network $Q$ with $n$ nodes,
$n\geq m,$ labeled $1$ to $n.$ While we could do this analysis on a general
network with arbitrary arc lengths, we take a graph theoretic assumption where
all arc lengths are $1,$ so we will from now on call $Q$ a graph. In this
section we describe the dynamic model that we use throughout the paper. We do
this in three stages: States, Motion of agents, the resulting Markov chain.
Since we will restrict the arc length to 1 here (although other lengths could
be considered) we will henceforth use the term \textit{graph} instead of
\textit{network}.

\subsection{States of the system}

There are several ways to denote the \textit{state} of the system. A general
way is to write square brackets $\left[  j_{1},j_{2},\dots,j_{n}\right]  ,$
where $j_{i}$ is the number of agents at node $i,$ with $%
{\textstyle\sum\nolimits_{i=1}^{m}}
j_{i}=m.$ We call the number $j_{i}$ the \textit{population} of node $i.$ We
could also use a notation $k_{j}$ which indicates the node that agent $j$ is
occupying, but in this paper we have no need to know this. Attached to every
state is a number $d$ denoting the minimum distance between two agents, where
we use the graph distance between nodes (the number or arcs in a shortest
path). For example, if we number the nodes of the line graph $L_{6}$
consecutively, and the state is $\left[  1,0,0,1,0,1\right]  ,$ then $d=2.$ If
a state has distance $d,$ it is called \textit{socially distanced} if $d\geq
D,$ where $D$ is a parameter of the problem. For example, the state $\left[
1,0,0,1,0,1\right]  $ is socially distanced for $D=1$ and $D=2$ but not for
$D=3.$ For the social distancing problem, the states with $d\geq D$ are
considered the absorbing states, because we want to calculate the expected
time to reach one of them, and the expected time to absorption is a standard
problem for Markov chains. For other problems (gathering, search game), we
have different absorbing states. The set of all states, the state space, is
denoted $\mathcal{S}.$

\subsection{Lazy Random Walks}

The unifying idea of this paper is the use of agent motions of the following type.

\begin{definition}
A Lazy Random Walk (LRW) for an agent on the graph $Q,$ with laziness
parameter $p$ (and speed $q=1-p$) is as follows. With probability $p,$ stay at
your current node. With probability $q=1-p$, go equiprobably to any of the
$\delta$ adjacent nodes, where $\delta$ is the degree of your current node. If
$p=0$ this is called simply a Random Walk. If the graph has constant degree
$\Delta,$ then an LRW with $p=1/\left(  \Delta+1\right)  $ then the process is
called a Loop-Random Walk. That is because it would be a Random Walk if loops
were added to every node. That is, all adjacent nodes are chosen equiprobably,
including the current node.
\end{definition}

For various problems considered in the paper, random walks or loop-random
walks will be optimal, in terms of minimizing the mean time to reach the
desired state.

If all the agents in the model follow independent LRWs with the same value of
$p,$ i.e. our main assumption, then a Markov chain is thereby defined on the
state space $\mathcal{S}$. We only consider triples $m$ (number of agents),
$D$ desired distancing and $Q$ (the connected graph), where it is possible for
have distanced states. For example the triple $m=3,$ $D=2$ and $Q=C_{5}$
(cycle graph with 5 nodes) has no distanced states. In general, we assume that
the $D-$Independence number (maximum number of $D$ distanced nodes) is at
least $m.$ If $D=1$ this is called simply the independence number. If $n=m$
and $D=1$ we call this the spatial dispersion problem of Alpern an Reyniers
(2002), an important special case of social distancing.

Given $Q$ (with $n$ nodes), $m$ and $D,$ there is a Markov chain on the state
space $\mathcal{S}$ with a non-empty set of absorbing states $\mathcal{A}$.
Suppose we number the non-absorbing states as $1,2,\dots,N,$ and let $B $
denote the $N\times N$ matrix where $b_{i,j}$ is the transition probability
from state $i$ to state $j.$ Let $t$ be the vector $\left(  t_{i}\right)  $
denote the expected time (number of transition steps) to reach an absorbing
state from state $i.$ The $t_{i}$ then satisfy the simultaneous equations%
\begin{align}
t_{1}  &  =1+b_{11}t_{1}+\dots+b_{1j}t_{j}+\dots+b_{1n}t_{n}\label{simul}\\
&  \vdots\nonumber\\
t_{i}  &  =1+b_{i1}t_{1}+\dots+b_{ij}t_{j}+\dots+b_{in}t_{n}\nonumber\\
&  \vdots\nonumber\\
t_{n}  &  =1+b_{n1}t_{1}+\dots+b_{nj}t_{j}+\dots+b_{nn}t_{n}\nonumber
\end{align}
We can write this in matrix terms, where $J_{n}$ the $1$ by $n$ matrix of $1$s
and $I_{N}$ is the $N\times N$ identity matrix, as%

\begin{align*}
t  &  =J_{n}+Bt,\text{ or}\\
\left(  I_{n}-B\right)  t  &  =J_{n},\text{ with solution}\\
t  &  =\left(  I_{N}-B\right)  ^{-1}~J_{n}.
\end{align*}
So the solution for the absorption time vector $t$ is given by%
\begin{equation}
t=FJ_{n},\text{ where} \label{t}%
\end{equation}%
\begin{equation}
F=\left(  I_{N}-B\right)  ^{-1}~J_{n}\text{ is known as the
\textit{fundamental matrix}.} \label{F}%
\end{equation}
In our model the Markov chain has a parameter $p,$ so these times
$t_{i}\left(  p\right)  $ will depend on $p.$ This use of the fundamental
matrix to calculate absorption times (\ref{F}) is well known. For example see
Section 8 of Kemeny, Snell and Thompson (1974). We use this formula (\ref{F})
often in this paper, starting in Section 3.1. In Section 4.2 we do the same
calculation using an equivalent method with the original simultaneous equations.

In some applications (e.g. two searchers for one hider) we wish to know the
probability that a particular absorbing state is reached (which searcher finds
the hider). Formulae for this problem are also known, but in the event we find
a more direct way to determine this. This will be made clear in Section 4.4.

A useful variation is to allow agents to see the number $k$ of agents at their
node, the \textit{population} of the node. In this case may allow a laziness
$p_{k}$ that depends on this $k.$ In most cases we consider the problem of
finding the laziness value $p=\bar{p}$ which minimizes the absorption time.
However in search game models considered in Sections 4.3 and 4.4, the hider
wishes to maximize the expected absorption time while the searcher wishes to
minimize it. We will define the gathering and search game models when they are
introduced, respectively in Sections 4.2 and 4.3. There are also cases where
individual agents do not have the same goal, for example in Section 6.

A useful variation is to allow agents to see the number $k$ of agents at their
node, the \textit{population} of the node. In this case may allow a laziness
$p_{k}$ that depends on this $k.$ For example if I find myself at a node with
three other agents, I stay there with probability $p_{4},$ which is a number
that is part of the overall strategy. But generally, and unless stated, we
assume there is only one value of $p$ regardless of the population of the node.

\section{ Social Distancing on $C_{n},n\geq4$ with $D=2$}

Generally we will consider problems with at least $m=3$ agents on a graph, but
to illustrate the main concepts of the paper we begin with a simple example
where two agents who start in adjacent nodes try to achieve distance $D=2$ on
a cycle graph $C_{n}$ with $n\geq4$ nodes. It turns out that the cases
$n\geq5$ and $n=4$ have different solutions. We take advantage of the symmetry
of the cycle graph to use a reduced state space determined by the distance $j$
between the agents. State $j$ covers all configurations where this distance is
$j-1,$ so that we have the usual row and column numbers for our matrices. The
three states $j=1,2,3$ are depicted in Figure 1 for both $C_{5}$ (top) and
$C_{4}$ bottom. For both cases of $n,$ there are (up to symmetry) two
non-absorbing states ($1$ and $2$) and a single absorbing state $3.$%
\begin{figure}[H]
    \centering
    \includegraphics[width=0.85\textwidth]{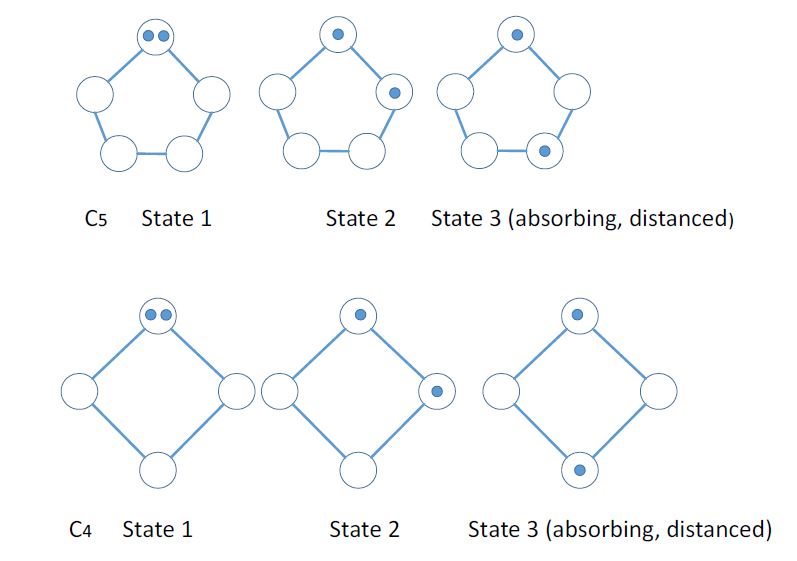}
    \caption{States for $C_5$ and $C_4$, $d=2$}
    \label{fig:my_label}
\end{figure}
To see the difference between $n\geq5$ and $n=4,$ consider the (expected)
absorption time $T$ from state $2$ when adopting a random walk (a LRW with
$p=0$). In $C_{n},$ $n\geq5,$ when both agents move from state $2,$ if they go
in the same direction (probability $1/2)$ or towards each other (probability
$1/4)$, they stay in state 2. If they go in opposite directions (probability
$1/4)$ they reach the absorbing state 3. So $T$ satisfies the equation%
\[
T=\left(  3/4\right)  \left(  1+T\right)  +\left(  1/4\right)  \left(
1\right)  \Longrightarrow\text{ }T=4.
\]
However in the graph $C_{4},$ if they start in state 2, they stay forever in
state $2,$ so $T=\infty.$ In the following two subsections on $n\geq5$ and
$n=4,$ we consider Population Dependent Lazy Random Walks, using the notation
$p_{1}=p$ (used when alone at a node - in state 2) and $p_{2}=r.$ We set
$q=1-p$ and $s=1-r$ for the complementary probabilities. We solve this problem
and then the simpler LRW problem by setting $p_{1}=p_{2}$ ( $p=r $).

\subsection{The case of $C_{n},$ $n\geq5$}

As illustrated in Section 2, we only need to calculate the transition
probabilities between the non-absorbing states. Here these are $1$ and $2.$
This transition matrix is given by%
\begin{equation}
B=\left(
\begin{array}
[c]{cc}%
r^{2}+s^{2}/2 & 2rs\\
pq & p^{2}+\left(  3/4\right)  q^{2}%
\end{array}
\right)  ,\text{ } \label{B}%
\end{equation}
and the fundamental matrix by,%
\begin{equation}
F=\left(  I_{2}-B\right)  ^{-1}=F=\frac{1}{\alpha}\left(
\begin{array}
[c]{cc}%
\frac{2\left(  1+7p\right)  }{1-r} & \frac{16r}{1-p}\\
\frac{8p}{1-r} & \frac{4\left(  1+3r\right)  }{1-p}%
\end{array}
\right)  ,\text{where} \label{F1}%
\end{equation}%
\[
\alpha=\left(  1+3r+p(7+5r)\right)
\]
$\allowbreak\allowbreak$

$\allowbreak$The absorption times from states $i=1,2$ are%
\[
\left(
\begin{array}
[c]{c}%
t_{1}\left(  p,r\right) \\
t_{2}\left(  p,r\right)
\end{array}
\right)  =F\left(
\begin{array}
[c]{c}%
1\\
1
\end{array}
\right)  =\frac{1}{\alpha}\left(
\begin{array}
[c]{c}%
\frac{14p+2}{\left(  1-r\right)  }+16\frac{r}{\left(  1-p\right)  }\\
\frac{12r+4}{\left(  1-p\right)  }+8\frac{p}{\left(  1-r\right)  }%
\end{array}
\right)
\]
$\allowbreak$ Since we are starting in state $2$ we minimize $t_{2}\left(
r,p\right)  $ at
\begin{align*}
\bar{p}_{2}  &  \equiv\bar{r}=0,~\bar{p}_{1}=\bar{p}=\left(  \sqrt
{33}-5\right)  /2\simeq\allowbreak0.372\,28,\\
~\bar{t}_{2}  &  \equiv t_{2}\left(  \bar{p}_{2},\bar{r}_{2}\right)  =\left(
\sqrt{33}+15\right)  /8\simeq\allowbreak2.\,\allowbreak593\,1
\end{align*}
So when acting optimally, the agents always move when they are with another
agent, and move about 63\% of the time when they are alone. The absorption
time of about $2.6$ periods is considerably less than the $4$ periods they
take if they both follow a random walk.

If the agents must move according to a common LRW because they are unaware of
the local population, we seek a minimum absorption time subject to $r=p.$
\[
t_{2}\left(  p,p\right)  =\frac{20p+4}{\left(  1-p\right)  \left(
10p+5p^{2}+1\right)  },\text{ }%
\]
with minimum of $t_{2}\left(  1/5,1/5\right)  =\allowbreak25/8\simeq
\allowbreak3.\,\allowbreak125$ at $p=1/5.$ So even without being aware of the
population of their location, they can still do a bit better than the random
walk ($p=0$) absorption time of $4$. Similar results can be obtained for
starting together (state 1) or starting randomly.

\subsection{The case of $C_{4}$}

On the graph $C_{4},$ the transition matrix changes in the transition
probability from state 2 to state 2 because if the agents move away from each
other the state remains state 2. The transitions among the non-absorbing
states are now
\[
B=\left(
\begin{array}
[c]{cc}%
r^{2}+s^{2}/2 & 2rs\\
pq & p^{2}+q^{2}%
\end{array}
\right)
\]

A similar analysis to that for $n\geq5$ now shows that starting from either
state 1 or state 2, the optimal strategies are $p_{2}=r=0$ and $p_{1}=p=1/2.$
Assuming this, we have $\bar{t}_{1}=2$ and $\bar{t}_{2}=3.$ This is
counter-intuitive in that it is quicker to socially distance starting with
both agents at the same node than starting with them at adjacent nodes. If we
seek the optimal LRW, the solution depends on where we start. If we start at
state 2 (two at a node), then it turns out that the random walk ($p=0)$ is
optimal, with (as shown above) an absorption time of 4. We already know that a
random walk starting at state 2 will never achieve social distancing, as in
this case state 2 will never be left. In this case the optimal $p$ is $\left(
1/10\right)  (-1+(49-20\sqrt{6})^{1/3}+(49+20\sqrt{6})^{1/3})=\allowbreak
0.382\,72.$ The absorption time for this LRW is approximately $4.45$.

Although this example is very simple, with only two agents, it illustrates the
use of population dependent strategies. That is, letting agents have awareness
of their immediate environment. It also shows why random walks, which maximize
the speed of the agents, do not necessarily lead to the quickest dispersal times.

\section{ Three Agents on the Cycle Graph $C_{3}$}

Two important classes of graphs are the cycle graphs $C_{n}$ and the complete
graphs $K_{n},$ which coincide for $n=3$ nodes. Due to the symmetry of the
graph we can use a special notation for states, rather than the more general
one defined earlier in Section 2. The problem is small enough for us to obtain
exact solutions, whereas the larger graphs will be studied later using
simulation. For the first two results, on dispersion (social distancing) and
gathering (multiple rendezvous) of three agents on $C_{3},$ we define three
states $j=1,2,3$ as those where the agents lie on $j$ distinct nodes. The
third result, on the search game, will require a different notion of states.

\subsection{Social distancing on $C_{3}$}

We first consider how three agents placed on the nodes of $C_{3}$ can achieve
social distancing with $D=1.$ This means that all pairwise distances must be
at least $1,$ that is, the agents must occupy distinct nodes. This is also
called the dispersion problem (one agent at each node). It turns out,
surprisingly, that the initial placement of the agents does not affect the
optimal strategy, which is the loop-random walk.

\begin{proposition}
If three agents are placed in any way on the nodes of $C_{3},$ then the
expected time to the social distanced state $j=3$ (one on each node) is
uniquely minimized when the agents adopt the loop-random walk $\left(
p=1/3\right)  $.
\end{proposition}

\begin{proof}
If all agents adopt the same laziness $p$ (speed $q=1-p),$ the transition
matrix for the non-absorbing states $1$ (all at same node) and $2$ (two at one
node, one at another) is given by%
\[
B=\left(
\begin{array}
[c]{cc}%
p^{3}+q^{3}/4 & 3p^{2}q+3pq^{2}/2+3q^{3}/4\\
p^{2}q/2+pq^{2}/4+q^{3}/8 & p^{3}+3p^{2}q/2+9pq^{2}/4+5q^{3}/8
\end{array}
\right)  .
\]
Using the fundamental matrix $F=\left(  I-B\right)  ^{-1},$ with $I$ the
identity matrix of size $2,$ we obtain the expected times $t_{j}$ from state
$j$ to the absorbing state $3$ as%
\[
\left(
\begin{array}
[c]{c}%
t_{1}\left(  p\right) \\
t_{2}\left(  p\right)
\end{array}
\right)  =\left(  I-B\right)  ^{-1}\left(
\begin{array}
[c]{c}%
1\\
1
\end{array}
\right)  =\left(
\begin{array}
[c]{c}%
\frac{6}{1+2p-3p^{2}}\\
\frac{2(7+12p+9p^{2})}{3+6p+18p^{3}-27p^{4}}%
\end{array}
\right)  .
\]
$\allowbreak\allowbreak$ To minimize $t_{1},$ we calculate
\[
t_{1}^{\prime}\left(  p\right)  =\left(  p-1/3\right)  \frac{36}{\left(
3p+1\right)  ^{2}\left(  p-1\right)  ^{2}},
\]
to observe that $t_{1}$ is decreasing for $p<1/3$ and increasing for $p>1/3$
and hence has a unique minimum at $p=1/3.$ Similarly the time to the absorbing
state from state $2$ is given by
\[
t_{2}\left(  p\right)  =\frac{2(7+12p+9p^{2})}{3+6p+18p^{3}-27p^{4}}.
\]
By calculating%
\begin{align*}
t_{2}^{\prime}\left(  p\right)   &  =\left(  p-1/3\right)  \frac{12}{3}%
\frac{1-6p+36p^{2}+54p^{3}+27p^{4}}{\left(  p-1\right)  ^{2}\left(
3p+1\right)  ^{2}\left(  3p^{2}+1\right)  ^{2}}\\
&  =\left(  p-1/3\right)  \left[  \frac{12}{3}\frac{\left(  1-3p\right)
^{2}+27p^{2}+54p^{3}+27p^{4}}{\left(  p-1\right)  ^{2}\left(  3p+1\right)
^{2}\left(  3p^{2}+1\right)  ^{2}}\right]  ,
\end{align*}
and observing that the bracketed expression is positive on $\left(
0,1\right)  ,$ we see as above that $t_{2}$ has a unique minimum at $p=1/3.$
$\allowbreak$ Since $3=1+\delta,$ where $\delta=2$ is the degree of (every
node of) $C_{3},$ we see that this is the loop-random walk.
\end{proof}

\subsection{Gathering (multi-rendezvous) on $C_{3}$}

The Rendezvous Problem (Alpern, 1995), asks how two mobile agents who do not
know the location of the other, can meet in least expected time, called the
\textit{Rendezvous Value} of the problem. We now a multiple agent version of
that problem. Consider the \textit{gathering}, or \textit{multiple sticky
rendezvous problem}, where agents who meet merge into a single agent and the
aim is to have all agents at the same node. We consider the symmetric version
of the problem, where all agents must adopt the same strategy. In the present
context this means they all adopt the same laziness $p$ in their LRW. This has
previously been considered (see Section 5 of Alpern (1995) ) only for simple
two-agent rendezvous. Here the absorbing state is state 1, where the agents
together occupy 1 node. The sticky version for multiple agents was studied for
agents on a line graph, in Baston (1999). Again, our result is surprising in
that the initial placement of the agents on $C_{3}$ does not affect the solution.

\begin{proposition}
If three agents are placed in any way on $C_{3}$ then the unique solution to
the gathering problem is the loop-random walk, $p=1/3$ in this case. The
Rendezvous Value for the problem starting from state $2$ is $3$ and from state
3 is $27/7.$
\end{proposition}

\begin{proof}
If, as required, all agents adopt the same laziness $p$ (speed $q=1-p),$ the
transition matrix for the non-absorbing states $2$ (two at one node, one at
another) and $3$ (all at different nodes) is given by%
\[
B=\left(
\begin{array}
[c]{cc}%
p^{2}+p(1-p)+\frac{3(1-p)^{2}}{4} & 0\\
3p^{2}(1-p)+\frac{3p(1-p)^{2}}{2}+\frac{3(1-p)^{3}}{4} & p^{3}+\frac
{3p(1-p)^{2}}{4}+\frac{(1-p)^{3}}{4}%
\end{array}
\right)  .
\]
So by the general formulae (\ref{t}) and (\ref{F1}), the fundamental matrix is
given by%
\[
F=\left(
\begin{array}
[c]{cc}%
\frac{4}{1+2p-3p^{2}} & 0\\
\frac{4+12p^{2}}{1+3p+p^{2}+p^{3}-6p^{4}} & \frac{4}{3+3p^{2}-6p^{3}}%
\end{array}
\right)  ,\text{ with }t=\left(  t_{2},t_{3}\right)  \text{ given by}%
\]%
\[
\left(
\begin{array}
[c]{cc}%
\frac{4}{1+2p-3p^{2}} & 0\\
\frac{4+12p^{2}}{1+3p+p^{2}+p^{3}-6p^{4}} & \frac{4}{3+3p^{2}-6p^{3}}%
\end{array}
\right)  \left(
\begin{array}
[c]{c}%
1\\
1
\end{array}
\right)  =\left(  \frac{4}{-3p^{2}+2p+1},\frac{4\left(  4+3p+9p^{2}\right)
}{3\left(  1+3p+p^{2}+p^{3}-6p^{4}\right)  }\right)
\]
It is clear that $t_{2}$ is minimized where the denominator is maximized,
where $-6p+2=0,$ $p=1/3,$ with $t_{2}\left(  1/3\right)  =3.$ Similarly
$t_{3}$ is minimized when
\begin{align*}
d\left(  p\right)   &  =\left(  \frac{4\left(  4+3p+9p^{2}\right)  }{3\left(
1+3p+p^{2}+p^{3}-6p^{4}\right)  }\right)  ^{\prime}\\
&  =\frac{4\left(  3p-1\right)  \left(  17p+39p^{2}+27p^{3}+36p^{4}+9\right)
}{3\left(  -6p^{4}+p^{3}+p^{2}+3p+1\right)  ^{2}}=0,\text{ giving}\\
p  &  =1/3\text{ and }t_{3}\left(  1/3\right)  =\frac{27}{7}%
\end{align*}

As $3=1+2=1+\Delta$, where $\Delta=2$ is the degree of all nodes of $C_{3}, $
the LRW with $p=1/3$ is the loop-random walk.
\end{proof}

In this and larger gathering problems, every state has some number $k$ of
occupied nodes, those any any such nodes being considered glued together and a
single new agent. Note that the set of states $S_{K}$ with $k\leq K$ for some
$K$ is an invariant, or absorbing \textit{set}. This means that we can find
expressions for those $t_{i}\left(  p\right)  $ for $i$ in $S_{2}$ first, then
use this to find $t_{i}\left(  p\right)  $ for $i$ in $S_{3},$ and so on. This
is just a matter of a particular way of solving the simultaneous equations in
(\ref{simul}) in a recursive way. For example in the two state problem of this
section, we first solve for $t_{2}$ in $t_{2}=b_{2,2}\left(  1+t_{2}\right)  $
and then for $t_{3}$ in $t_{3}=b_{3,2}\left(  1+t_{2}\right)  +b_{3,3}\left(
1+t_{3}\right)  ,$ where the rows are considered row 2 and row 3. We consider
this as recursively solving for the variables in the simultaneous equations
rather than as dynamic programming because we cannot optimize the values of
$p$ for small $K$ and then use these values for larger $K.$ We are allowed
only a single value of $p$ and in addition agents do not themselves know the
current value of $k.$

However there is a variation of the gathering problem on $C_{3}$ which could
be solved with dynamic programming, as suggested by an anonymous referee. In
the current model, when two agents meet, the remaining agent is unaware of
this and hence must continue with an unchanged strategy $p,$ so he would not
be aware he was in a solved case. Suppose we consider a \textit{different}
model in which a central controller sends out a signal to all agents telling
how many new agents $k$ there now are, considering gluing of those who have
met. For $C_{3}$ the distribution of agents on $C_{3}$ (the state) is
determined by $k.$ \ Suppose we let the agents choose a common value of $p$
that depends on $k,$ call it $_{k}p$ . In that case, we could first minimize
$T_{2}$ for some \ $_{2}\bar{p}$\ and then solve the $k=3$ problem by using
\ $_{2}\bar{p}$ when two agents meet. However even with this intervention
approach we could not solve the general gathering problem of $m$ agents
randomly placed on $C_{n}$ because after two meet the $m-1$ new agents would
not be randomly placed. (We also note that for the particular case of three
agents on $C_{3}$ the new problem with added information does not lead to a
different answer, as all the optimal values of $p$ are the same, 1/3. But it
would be a different method.)

\subsection{Two searcher team and one hider on $C_{3}$}

We now consider a search game played by two mobile searchers and a mobile
hider on $C_{3}.$ These games were introduced by Isaacs (1965) and studied
initially by Zelikin (1972), Alpern (1974) and Gal (1979). For a comprehensive
treatment, see Gal (1980) and Alpern and Gal (2003). We place the three agents
on $C_{3}$ randomly. The searchers choose a common laziness $s$ and the hider
chooses a laziness $h.$ In this instance, we take the point that the searchers
are a team, mother and father to a hungry infant. They have the common aim of
minimizing the time $T$ taken to find the hider, who wants to maximize $T.$
Here $T$ is the first time that one of the searchers finds the hider, it does
not matter which searcher it is. (We could introduce competition between the
searchers, but we shall not do so here.) It is not clear \textit{a priori}
that there will be a saddle point. However in the event, we show that there is
one, with $h$ about $.51$ and $s$ about $.28$. Thus the searcher moves more
frequently than the hider. Ruckle (1983) has considered this problem on
$C_{n}$ (cycle graph with $n$ nodes) when there is a single searcher and a
single hider.

There are four states (up to symmetry, as usual): States 1 and 2 are non
absorbing (hider is not caught), states 3 and 4 are absorbing (the hider has
been caught). See Figure 2. A random initial placement results in these states
occurring with respective probabilities $2/9,$ $2/9,$ $4/9$ and $1/9.$%
\begin{figure}[H]
    \centering
    \includegraphics[width=0.88\textwidth]{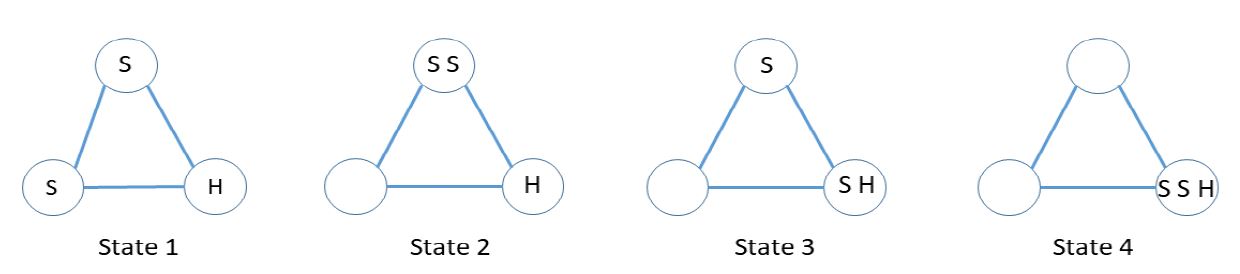}
    \caption{Four states of search game.}
    \label{fig:my_label}
\end{figure}
To calculate the expected value of the capture time $t_{j}$ (number of periods
to absorption) from state $j=1,2,$ it is only necessary to know the transition
probabilities between these two states, which are given by the following
$2\times2$ matrix $B$ (where $h^{\prime}=1-h$ and $s^{\prime}=1-s$).%
\begin{align*}
B  &  =\left(
\begin{array}
[c]{cc}%
\frac{1}{4}h^{\prime}s^{\prime2}+\frac{1}{4}hs^{\prime2}+\frac{1}{2}h^{\prime
}ss^{\prime}+hs^{2} & \frac{1}{4}h^{\prime}s^{\prime2}+\frac{1}{4}h^{\prime
}ss^{\prime}+hss^{\prime}\\
\frac{1}{4}h^{\prime}s^{\prime2}+\frac{1}{4}hs^{\prime2}+\frac{1}{2}h^{\prime
}ss^{\prime}+hss^{\prime} & \frac{1}{2}h^{\prime}s^{\prime2}+\frac{1}%
{4}hs^{\prime2}+\frac{1}{2}h^{\prime}s^{2}+hs^{2}%
\end{array}
\right) \\
&  =\frac{1}{4}\left(
\begin{array}
[c]{cc}%
1-s^{2}-2hs+6hs^{2} & 1-h-s+5hs-4hs^{2}\\
1+2hs-s^{2}-2hs^{2} & 2-h-4s+2hs+4s^{2}+hs^{2}%
\end{array}
\right)  ,\text{ with fundamental matrix}\\
F  &  =\left(  I-B\right)  ^{-1}=\frac{4}{E}\left(
\begin{array}
[c]{cc}%
-2-h-4s+2hs+4s^{2}+hs^{2} & -1+h+s-5hs+4hs^{2}\\
-1-2hs+s^{2}+2hs^{2} & -3-2hs-s^{2}+6hs^{2}%
\end{array}
\right)  ,\text{ with }E=
\end{align*}%
\[
-5-4h-13s+9hs-4h^{2}s+9s^{2}-hs^{2}+22h^{2}s^{2}-3s^{3}+31hs^{3}-28h^{2}%
s^{3}+4s^{4}-19hs^{4}+2h^{2}s^{4}%
\]
As in previous analyses, we then get the absorption times as%
\begin{align*}
\left(
\begin{array}
[c]{c}%
t_{1}\\
t_{2}%
\end{array}
\right)   &  =F\left(
\begin{array}
[c]{c}%
1\\
1
\end{array}
\right)  ,\text{ and expected meeting time}\\
T\left(  h,s\right)   &  =\left(
\begin{array}
[c]{cc}%
2/9 & 2/9
\end{array}
\right)  \left(
\begin{array}
[c]{c}%
t_{1}\\
t_{2}%
\end{array}
\right)  =\frac{a}{b},\text{ where }a\text{ is}%
\end{align*}%
\[
{\small 8}\left(  -7-3s-7hs+4s^{2}+13hs^{2}\right)  \text{ and }b\text{ is}%
\]%
\begin{gather*}
{\small -45-36h-117s+81hs-36h}^{2}{\small s+81s}^{2}{\small -9hs}%
^{2}{\small +198h}^{2}{\small s}^{2}\\
{\small -27s}^{3}{\small +279hs}^{3}{\small -252h}^{2}{\small s}%
^{3}{\small +36s}^{4}{\small -171hs}^{4}{\small +18h}^{2}{\small s}^{4}%
\end{gather*}
To determine whether (and where) $T\left(  h,s\right)  $ has a saddle point,
we first find the critical points by solving the simultaneous equations
\[
T_{h}=T_{s}=0,\text{ where }T_{s}=\partial T/\partial s\text{ and }%
T_{h}=\partial T/\partial h.
\]
We can see that a solution exists by plotting the two curves in Figure 3. Note
that the curve $T_{h}=0$ appears to be close to a straight line.
    \begin{figure}[H]
     \centering
     \includegraphics[width=0.9\textwidth]{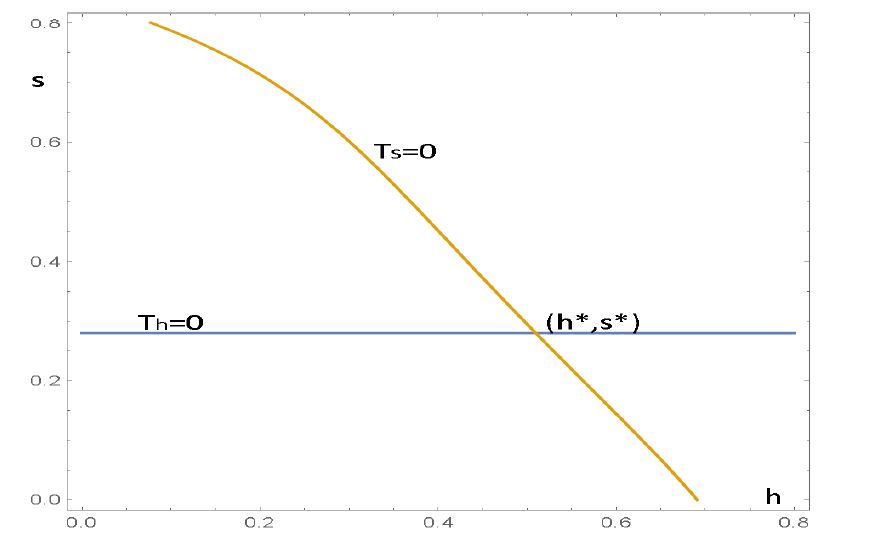}
     \caption{Intersection of $T_s=\partial T/\partial s$ and $T_h=\partial T/\partial s$ at $\left(  h^\ast,s^\ast\right)$.}
     \label{fig:my_label}
    \end{figure}
It is also useful to plot the optimal response curve $h=R\left(  s\right)$
of the hider for the function $T.$ We then can obtain $s^{\ast}$ exactly as
the solution to the fifth degree polynomial equation $T\left(  s,0\right)
=T\left(  s,1\right)  $ which simplifies to $14-15s-117s^{2}-33s^{3}%
-5s^{4}+60s^{5}=0$ and has a unique solution for $s\in\left[  0,1\right]  .$
    \begin{figure}[H]
     \centering
     \includegraphics[width=0.9\textwidth]{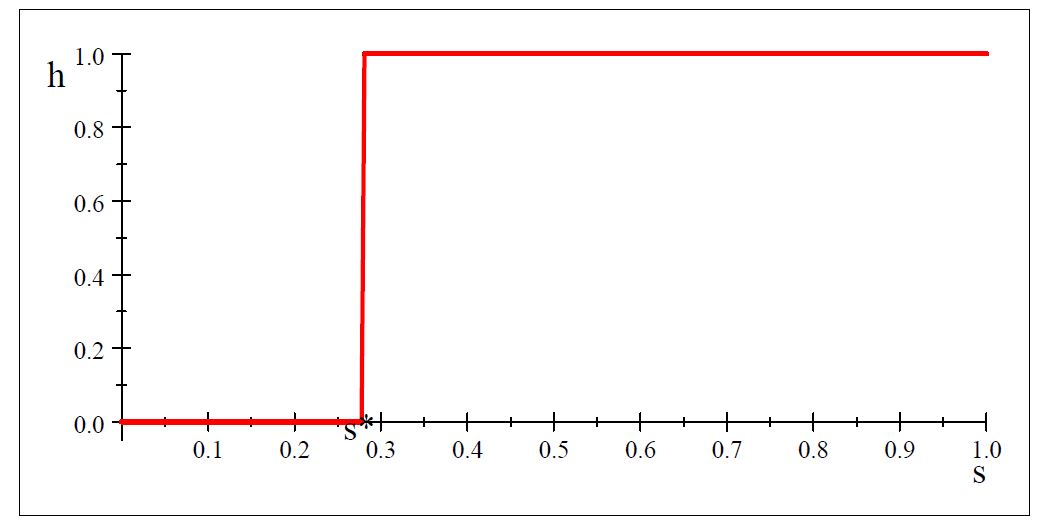}
     \caption{Hider's optimal response to s}
     \label{fig:my_label}
    \end{figure}
Numerical approximation of the critical point $\left(  h^{\ast},s^{\ast
}\right)  $ gives $h^{\ast}\simeq0.5097$ and $s^{\ast}\simeq0.2797,$ with game
value $V=T\left(  h^{\ast},s^{\ast}\right)  \simeq0.8390.$ To show that it is
a saddle point we approximate the determinant of the Hessian at about $-2.4,$
so it is certainly negative. But this fact is clearer from the Figure 5, which
shows plots where the horizontal axis can be $h$ or $s.$ The top (blue) curve
shows that the payoff $T$ is at least $V$ for any value of $s=x$ when the
hider adopts $h^{\ast}$ and is above $V$ if $s=x$ is not the optimal value
$s^{\ast}.$ The bottom (brown) curve shows that the searcher finds the hider
in time no more than $V$ when adopting $s^{\ast}.$ In this case the capture
time is not very sensitive to the value of $h.$%
    \begin{figure}[H]
     \centering
     \includegraphics[width=0.85\textwidth]{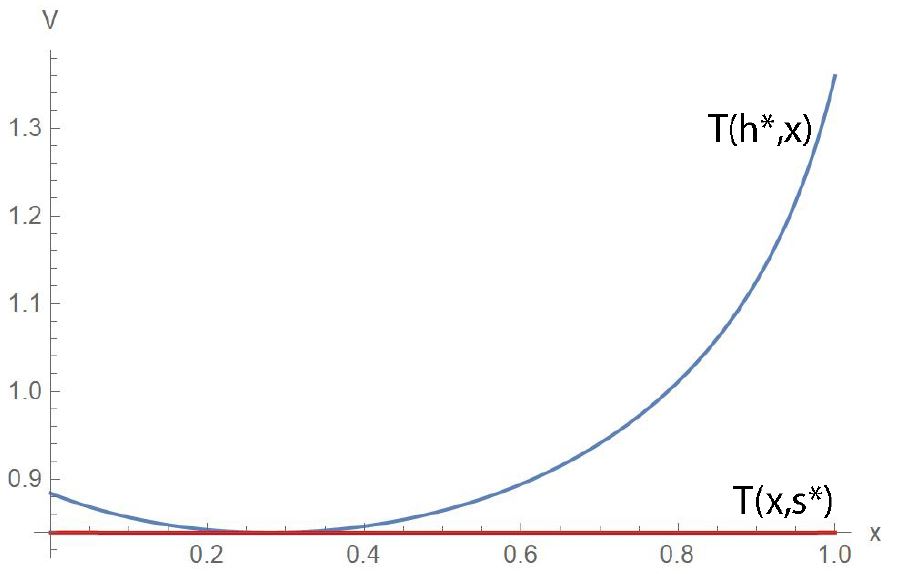}
     \caption{The pair $\left(  h^\ast,s^\ast\right)  $ is a saddle point of this
zero sum game.}
     \label{fig:my_label}
    \end{figure}

This analysis considers the two searchers as a team which wishes to minimize
the capture time $T.$ Perhaps a male and female who will bring the captured
prey back to their offspring, and it doesn't matter which one makes the kill.
A different approach (Payoff function) could model a competition between the
two searchers, as carried out in the next section.

\subsection{Competitive Search}

We now model the problem of two searchers and one hider as a three person
game, rather than considering the two searchers as a single player (team). As
in the previous subsection, the game ends at the first time $T$ when one or
both searchers coincide with the hider. The hider's payoff is simply $T.$ A
searcher gets payoff $1$ if he is the unique player to find the hider; $1/2 $
if both searchers find the hider at the same time and $0$ if the other
searcher finds the hider alone. This element of competition between the
searchers has been studied in Nakai (1986) and Duvocelle (2020), but here the
hider is also adversarial. Figure 6 shows the five states. States 1 and 2 are
non-absorbing, States 3, 4 and 5 are absorbing. Searcher 2 wins in State 3,
searcher 1 wins in State 4 and State 5 is a tie. The hider's payoff depends on
the time $T$ to reach an absorbing state.%
    \begin{figure}[H]
     \centering
     \includegraphics[width=0.95\textwidth]{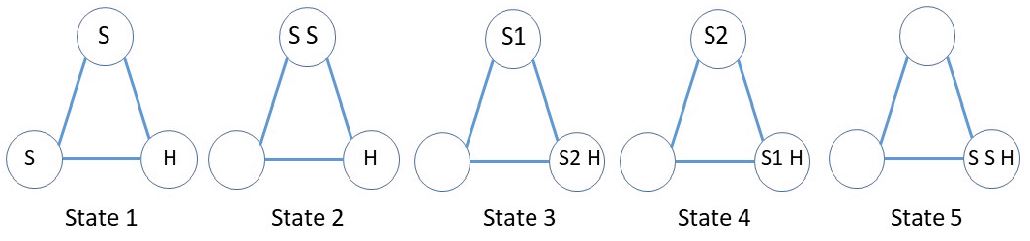}
     \caption{Five states in the competitive search game.}
     \label{fig:my_label}
    \end{figure}
We seek a Nash equilibrium that is symmetric with respect to the two
searchers. Denote the laziness of searcher 1 by $r$, searcher 2 by $s$ and of
the hider by $h.$ Let $E=E\left(  s_{1},s_{2},h\right)  $ now denote the
expected absorption time starting from a random state, respectively. Let
$a_{j},$ $j=3,4,5,$ denote the probability the game ends in State $j,$
assuming it starts randomly. Player 2's payoff is equal to $a_{3}+\left(
1/2\right)  ~a_{5},$ with a similar payoff for Player 1. We seek parameters $r
$ and $h$ such that

\begin{itemize}
\item $T\left(  r,r,h\right)  \geq T\left(  r,r,h^{\prime}\right)  $ for any
hider parameter $h^{\prime},$ and

\item $a_{3}\left(  r,r,h\right)  +\left(  1/2\right)  a_{5}\left(
r,r,h\right)  \geq a_{3}\left(  r,s,h\right)  +\left(  1/2\right)
a_{5}\left(  r,s,h\right)  ,$ for any $s.$
\end{itemize}

The probabilities that an absorbing Markov chain ends at each absorbing state
are easily calculated but we use a qualitative idea to avoid this calculation
on a five state chain. Instead we show that there is a dominating search
strategy (depending on $h$ but not the other search strategy) that ensures
always capturing in the next period with maximum probability. Such a strategy
clearly maximizes the searcher's payoff, regardless of what the other searcher
is doing.

To calculate the optimal response of the hider to a symmetric pair $\left(
s,s\right)  $ of strategies of the searchers we refer to Figure 4.

Suppose the hider adopts strategy $h.$ If a searcher can always maximize the
probability of finding the hider in the current period, he guarantees doing at
least as well as the other searcher. What is the best value of $s$ to maximize
this probability? If $h=1$ (hider stays still), then probability of capture is
$\left(  1-s\right)  /2.$ If $h=0$ (moves) the probability is $s/2+\left(
1-s\right)  /4.$ So against a general $h,$ the capture probability in the next
period is%
\begin{equation}
W\left(  s,h\right)  =h\left(  \left(  1-s\right)  /2\right)  +\left(
1-h\right)  \left(  s/2+\left(  1-s\right)  /4\right)  =\allowbreak\left(
\frac{1}{4}-\frac{3}{4}h\right)  s+\left(  \frac{1}{4}h+\frac{1}{4}\right)  .
\label{W}%
\end{equation}
The maximizing $s$ will be $1$ if $\left(  \frac{1}{4}-\frac{3}{4}h\right)  $
is positive, i.e. $h<1/3.$ The maximizing $s$ will be $0$ if $h>1/3.$ If
$h=1/3$ (loop-random walk) then all $s$ give the same capture probability.
Note that $s^{\ast}=1/3$ gives the searcher a loop-random walk, as $C_{3}$ has
degree $2$ for all nodes. We already showed in the previous subsection that if
$s=s^{\ast}$ then all $h$ give the same expected capture time from a random
start. So $h=1/3$ and $s=s^{\ast}\simeq.278$ give the searcher-symmetric
equilibrium $\left(  s,s,h\right)  =\left(  s^{\ast},s^{\ast},1/3\right)  .$

To see that this equilibrium is unique, suppose $s<s^{\ast}.$ Then as $h$ is
an optimal response to $\left(  s,s,\_\right)  ,$ we have that $h=0.$ In this
case we have in particular that $h<1/3$ so we showed above that the play of
each searcher to maximize the probability he finds the hider first is $s=1.$
This contradicts our assumption that $s<s^{\ast}.$ Similarly if $s>s^{\ast},$
then the best response is $h=1>1/3,$ so the maximizing $s$ if $0,$
contradicting our assumption.

\begin{proposition}
The unique searcher-symmetric Nash equilibrium to the competitive search game
on $C_{3}$ is given by the loop-random walk ($h=1/3)$ for the hider and a
laziness $s^{\ast}$ for both searchers, where $s^{\ast}\simeq0.2791$ is the
unique solution to the fifth degree polynomial equation $14-15s-117s^{2}%
-33s^{3}-5s^{4}+60s^{5}=0$ between $0$ and $1.$
\end{proposition}

To see why the team solution given in the previous subsection is not an
equilibrium with respect to the searchers, note that against $h=h^{\ast}%
\simeq.5097,$ a searcher playing $s=0$ (random walk) has a higher capture
probability in each period than one playing $s^{\ast}\simeq0.2791,$ as
$W\left(  0,.5097\right)  \simeq\allowbreak0.377\,4$ compared to $W\left(
s^{\ast},.5097\right)  \simeq\allowbreak0.340\,5,$ see (\ref{W}). Note that
the searchers behave the same at equilibrium whether or not they are working
as a team, but the hider moves more frequently when the hiders act as a team
rather than competitively.

\section{ Gathering and Dispersing on $C_{5},$ $m=3.$}

Suppose three agents are located on the cycle network $C_{n}.$ For this
section we take $n=5,$ but the following representation of states works for
all $n.$ We may use the symmetry of the network to reduce that state to three
numbers (actually $2,$ once we know $n).$ Let $j$ denote the distance between
the two closest agents and let $k$ denote the distance between the second
closest pair. Thus the arcs between the three agents have distances $j,$ $k$
and $n-j-k.$ For $n=5$ we have five states, as shown in Figure 7.%
Five states in the competitive search game.
  \begin{figure}[H]
     \centering
     \includegraphics[width=0.85\textwidth]{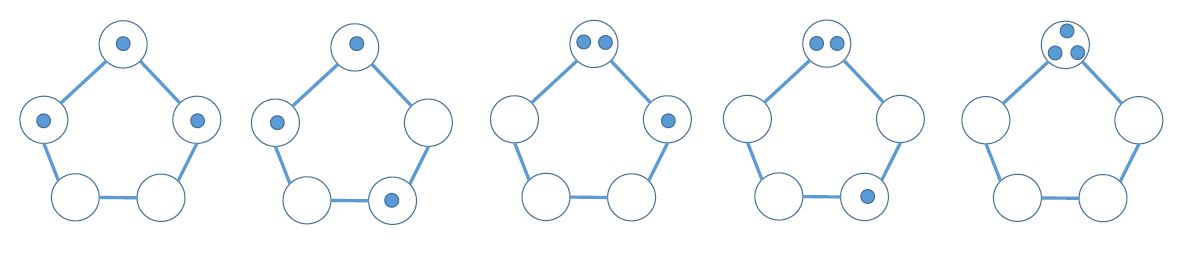}
     \caption{States $\left(  1,1\right)  ,\left(  1,2\right)  ,(0,1),(0,2),(0,0)$
left to right.}
     \label{fig:my_label}
    \end{figure}

In general, for three agents on $C_{n},$ we have a triangular set of states
$D_{3}=\left\{  \left(  j,k\right)  :0\leq j\leq k\leq\left(  n-j\right)
/2\right\}  .$ For the case $n=5$ considered here, the five states (in
$x=j,~y=k$ space) lie between the lines $j=k$ and $k\leq\left(  5-j\right)
/2, $ as shown as black disks in Figure 8.
    \begin{figure}[H]
     \centering
     \includegraphics[width=0.75\textwidth]{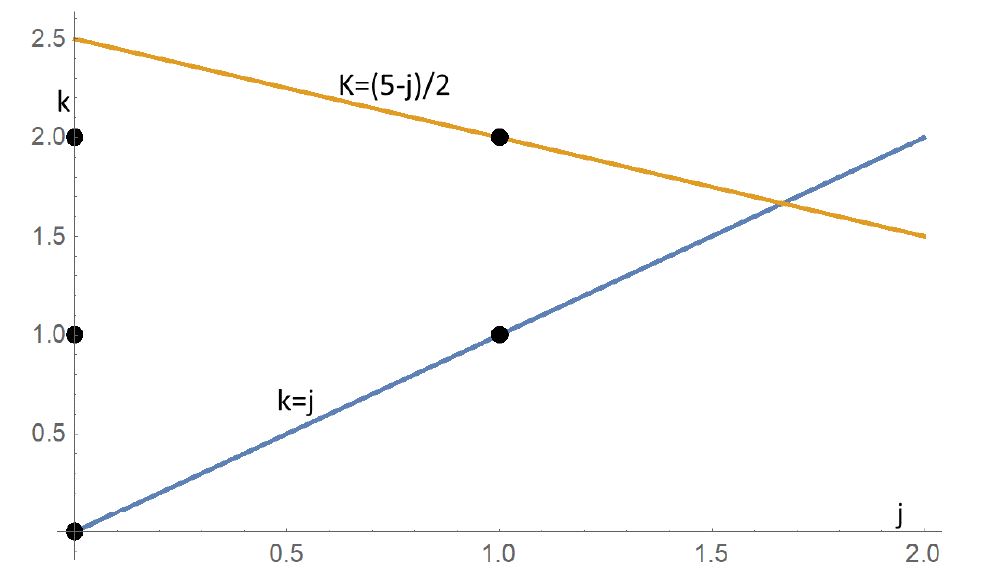}
     \caption{Five states for $C_5$ between $k=j$ and $k=\left(  5-j\right)  /2.$}
     \label{fig:my_label}
    \end{figure}
For larger values of $n,$ the states for three agents will be more numerous
and from state $\left(  j,k\right)  $ can transition to $\left(
j+x,k+y\right)  $ for $x,y\in\left\{  -2,-1,0,1,2\right\}  $ with some
exceptions. For example the nine states in the two extreme corners cannot be
reached (these circles are not filled in). See Figure 9. The state $\left(
j-2,k-2\right)  $ cannot be reached because if the two closest agents move
towards each other two cannot also move closer.%
    \begin{figure}[H]
     \centering
     \includegraphics[width=0.75\textwidth]{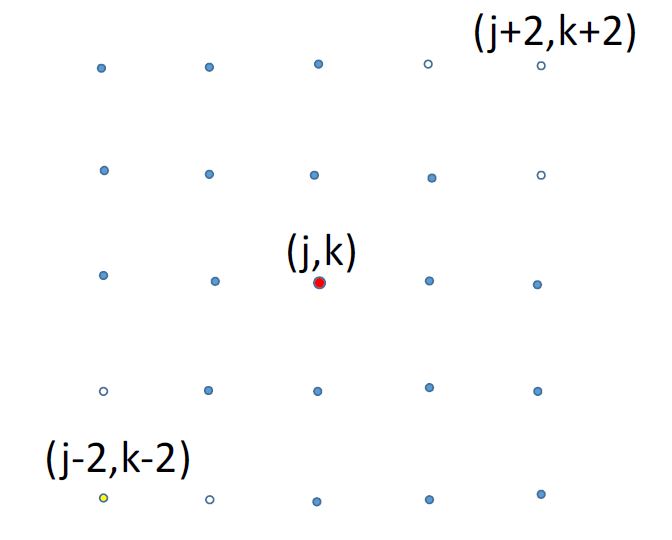}
     \caption{From central state $\left(  j,k\right)  $ 16 of the 25 states can be
reached.}
     \label{fig:my_label}
    \end{figure}
This figure indicates the complexity of analyzing even three agents for larger
cycle graphs and explains why will use simulation techniques to obtain
approximate solutions for larger cycle graphs.

\subsection{Gathering on $C_{5}$}

The gathering problem is defined in the same way on $C_{5}$ as earlier on
$C_{3}.$ The state $\left(  0,0\right)  $ in Figure 7 and 8 is the only
absorbing state and we number the other four from left to right. For example
state 1 is $\left(  1,1\right)  .$ The allowable transitions are shown in
Figure 10, where the letter labels are $a=(0,0),$ $b=\left(  0,2\right)  ,$
$c=\left(  0,1\right)  ,~d=\left(  1,1\right)  ,~e=\left(  1,2\right)  .$ The
non-absorbing states for our transition matrix are thus $e$ $\left(  1\right)
,$ $d\left(  2\right)  ,$ $c\left(  3\right)  $ and $b\left(  4\right)  .$%
    \begin{figure}[H]
     \centering
     \includegraphics[width=0.95\textwidth]{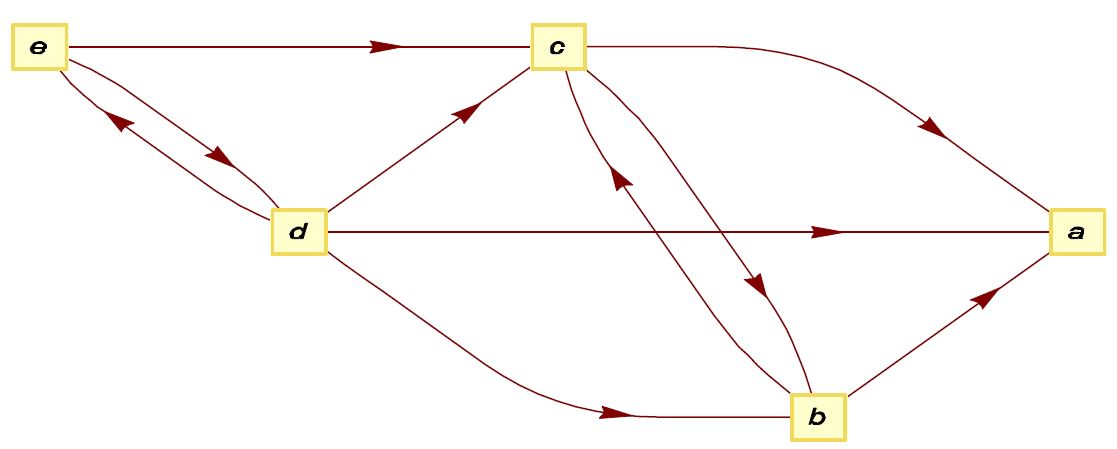}
     \caption{Allowable transitions for $C_5$ gathering.}
     \label{fig:my_label}
    \end{figure}
The non-absorbing states (rows) for our transition matrix are thus $e$
$\left(  1\right)  ,$ $d\left(  2\right)  ,$ $c\left(  3\right)  $ and
$b\left(  4\right)  ,$ and the $4$ by $4$ transition probability matrix for
these states, with all agents adopting $p$ (with $q=1-p$) is given by the $4$
by $4 $ matrix $B,$%

\[
\left(
\begin{array}
[c]{cccc}%
p^{3}+(pq^{2})/2+q^{3}/2 & p^{2}q+(3pq^{2})/4+q^{3}/4 & p^{2}q+(pq^{2}%
)/2+q^{3}/4 & p^{2}q+pq^{2}\\
p^{2}q+(3pq^{2})/4+q^{3}/4 & p^{3}+p^{2}q+(3pq^{2})/4+q^{3}/4 & (pq^{2}%
)/2+q^{3}/4 & p^{2}q+pq^{2}+q^{3}/4\\
0 & 0 & p^{2}+3q^{2}/4 & pq+q^{2}/4\\
0 & 0 & pq+q^{2}/4 & p^{2}+pq+q^{2}/2
\end{array}
\right)
\]

We then, as usual, calculate the fundamental matrix $F=\left(  I_{4}-B\right)
^{-1}$ and evaluate the times $t_{i}$ from state $i$ to absorption (gathering)
as%
\[
\left(
\begin{array}
[c]{c}%
t_{1}\\
t_{2}\\
t_{3}\\
t_{4}%
\end{array}
\right)  =F\left(
\begin{array}
[c]{c}%
1\\
1\\
1\\
1
\end{array}
\right)  =\left(
\begin{array}
[c]{c}%
\frac{-{\Large (8(9+58p+152p}^{2}{\Large +170p}^{3}{\Large +103p}%
^{4}{\Large +20p}^{5}{\Large )}}{{\Large (-1+p)(5+72p+274p}^{2}{\Large +420p}%
^{3}{\Large +349p}^{4}{\Large +140p}^{5}{\Large +20p}^{6}{\Large )}}\\
\\
\frac{{\Large -4(16+121p+331p}^{2}{\Large +385p}^{3}{\Large +249p}%
^{4}{\Large +50p}^{5}{\Large )}}{{\Large (-1+p)(5+72p+274p}^{2}{\Large +420p}%
^{3}{\Large +349p}^{4}{\Large +140p}^{5}{\Large +20p}^{6}{\Large )}}\\
\\
\frac{{\Large (12+20p)}}{{\Large (1+9p-5p}^{2}{\Large -5p}^{3}{\Large )}}\\
\\
\frac{{\Large (8+40p)}}{{\Large (1+9p-5p}^{2}{\Large -5p}^{3}{\Large )}}%
\end{array}
\right)  .
\]
  \begin{figure}[H]
     \centering
     \includegraphics[width=0.95\textwidth]{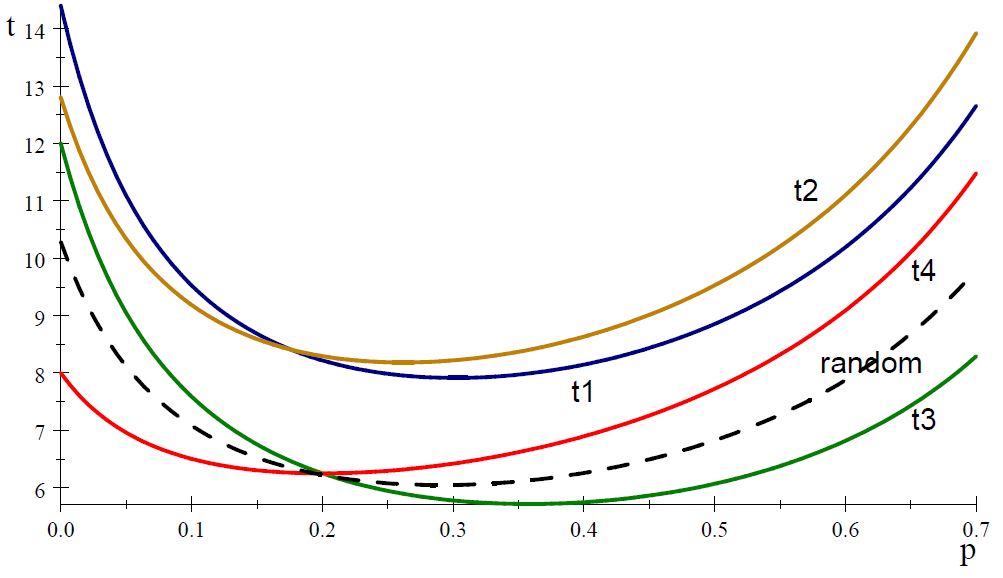}
     \caption{$t_{i}\left(  p\right)  $ starting from states $i=1,\dots,4$ and
random (dashed).}
     \label{fig:my_label}
    \end{figure}
The expected time to absorption for the different initial states are shown in
Figure 11 as functions of $p$ for $i=1$ blue, $2$ yellow, $3$ green, $4$ red.
The random start gives probabilities $\left(  6,6,6,6,1\right)  /25$ (the
final probability $1/25$ is for gathering right away, state 5).%

\begin{align*}
&
\begin{tabular}
[c]{|l|l|l|}\hline
initial state $i\backslash opt$ & $\bar{p}$ & $\bar{t}_{i}$\\\hline
$\left(  1,1\right)  =$ $\#1$ blue & 0.301 & 7.914\\\hline
$\left(  1,2\right)  =\#2$ yellow & 0.262 & 8.183\\\hline
$\left(  0,1\right)  =$ $\#3$ green & 0.358 & 5.716\\\hline
$\left(  0,2\right)  =\#4$ red & 0.200 & 6.250\\\hline
random & 0.283 & 6.794\\\hline
\end{tabular}
\\
&  \text{Table 1. Optimal }p\text{ for }C_{5}\text{ gathering.}%
\end{align*}

\bigskip

\subsection{Social distancing on $C_{5}$}

For the social distancing problem on $C_{5}$ with $m=3$ agents and $d=1$
(higher values of $d$ are not attainable on $C_{n}$, $n<6$) we have the same
five states as in Figure 7. However now the two states $\left(  j,k\right)  $
with $j=1$ are absorbing (distanced) because $j$ is by definition the minimum
pairwise distance between agents. We renumber the remaining states as
$S_{i}=\left(  0,i\right)  ,$ so that $S_{1}=\left(  0,0\right)  ,$
$S_{2}=\left(  0,1\right)  $ and $S_{3}=\left(  0,2\right)  .$ As usual we
only need to calculate the transition probabilities between non-absorbing
states, which are given (with $q=1-p$) by the $3$ by $3$ matrix%

\[
B=\left(
\begin{array}
[c]{ccc}%
p^{3}+q^{3}/4 & 3p^{2}q+3pq^{2}/2 & 3q^{3}/4\\
pq^{2}/4+p^{2}q/2 & p^{3}+pq^{2}+p^{2}q+3q^{3}/8 & p^{2}q/2+3pq^{2}%
/4+q^{3}/8\\
q^{3}/8 & p^{2}q/2+3pq^{2}/4+q^{3}/8 & p^{3}+p^{2}q/2+pq^{2}/4+q^{3}/2
\end{array}
\right)  .
\]
The times $t_{i}$ for absorption from $S_{i},$ shown in Figure 12, are given by%
\begin{gather*}
\left(
\begin{array}
[c]{c}%
t_{1}\\
t_{2}\\
t_{3}%
\end{array}
\right)  =\left(  I_{3}-B\right)  ^{-1}\left(
\begin{array}
[c]{c}%
1\\
1\\
1
\end{array}
\right)  =\\
\left(
\begin{array}
[c]{c}%
-(2(55+134p+316p^{2}+330p^{3}+45p^{4}))/(3(-1+p)(7+67p+126p^{2}+158p^{3}%
+75p^{4}+15p^{5}))\\
-(2(25+214p+232p^{2}+170p^{3}+15p^{4}))/(3(-1+p)(7+67p+126p^{2}+158p^{3}%
+75p^{4}+15p^{5}))\\
-(2(41+142p+152p^{2}+50p^{3}+15p^{4}))/(3(-1+p)(7+67p+126p^{2}+158p^{3}%
+75p^{4}+15p^{5}))
\end{array}
\right)
\end{gather*}%
 \begin{figure}[H]
     \centering
     \includegraphics[width=0.95\textwidth]{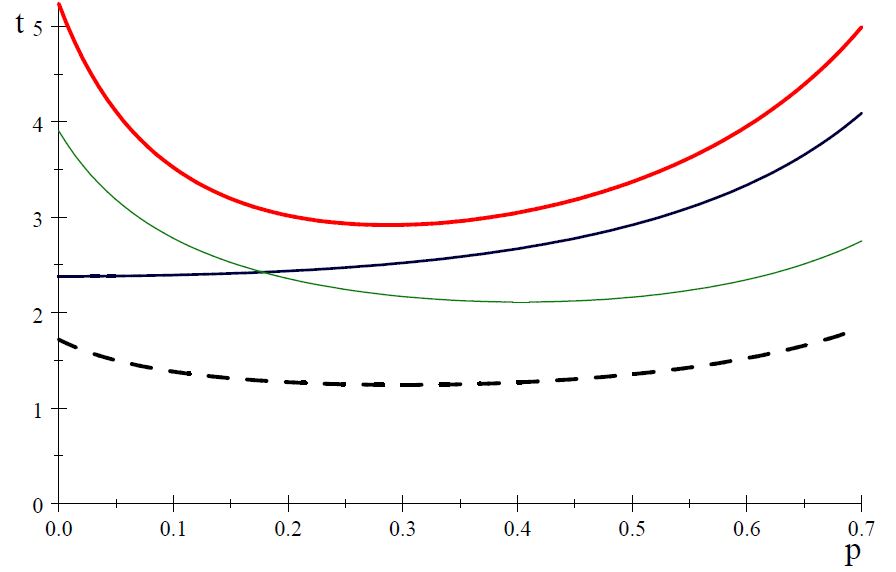}
     \caption{Times $t_{i},i=1$(red,thick), $2$(blue,medium), $3$(green,thin),
and random (dashed).}
     \label{fig:my_label}
    \end{figure}
\begin{align*}
&
\begin{tabular}
[c]{|l|l|l|}\hline
initial state $i\backslash opt$ & $\bar{p}$ & $\bar{t}_{i}$\\\hline
$\left(  0,0\right)  =$ $S_{1}$ red & $0.287$ & $2.\,\allowbreak918$\\\hline
$\left(  0,1\right)  =S_{2}$ blue & $0.006$ & $2.381$\\\hline
$\left(  0,2\right)  =$ $S_{3}$ green & $0.403$ & $2.\,\allowbreak111$\\\hline
\end{tabular}
\\
&  \text{Table 2. Optimal }p\text{ for }C_{5}\text{ Social Distancing.}%
\end{align*}
It is intuitive that social distancing takes the longest when the agents are
in the gathered position. When two are at the same location it takes longer to
disperse when the third is closest to them. The random starting process takes
a shorter time because there is already a high probability (13/25) that they
are dispersed, in which case the dispersal time is $0.$

\section{No Equilibrium in First-to-Disperse Game on $L_{3}$}

In this section we consider the game $G_{1}\left(  n\right)  ,$ where $n$
players start together at the end location $1$ on the line graph $L_{n}$ with
nodes $1,2,\dots,n.$ When some players first achieves "ownership" of a node
(are alone at their node), these players equally split a prize of $1.$ Each
player $i$ has a single strategic variable, her laziness probability $p_{i}.$
We seek symmetric equilibria (with all $p_{i}$ the same) for the cases
$n=2,3.$

We can consider this game as a selfish form of the social distancing problem
with $D=1$ and $m=n$ (so it is also a dispersion problem) on the line graph
$L_{n}.$ In a version of this problem with what we call
\textit{territoriality}, a player who is alone at her node becomes the
\textit{owner} of it. This means she stays there forever and anyone else who
lands there immediately moves away randomly in the next period.\textit{\ }So
the game considered here can be thought of as the beginning of a dispersal
problem with territoriality.

\subsection{The case $n=2$}

This is an almost trivial case. For any $p\in\left(  0,1\right)  ,$ the game
eventually ends with probability one (as soon as one player moves and one
stays, in the same period), with a payoff of 1/2, since both players will
achieve ownership at the same time. So \textit{any} pair $\left(  p,p\right)
$ is a symmetric equilibrium.

\subsection{The case $n=3$}

By symmetry, it is clear that when all players adopt stay probability $p,$
they all have expected payoff of $1/3.$ We will show that when any two players
adopt the same $p,$ the remaining player can get more than $1/3$ by a suitable
strategy, and hence there is no symmetric equilibrium. The algebra involved in
the proof is greatly simplified if we consider the "modified payoff" $M(q,p)$
to the single player (call her Player 1) adopting $q$ when the other two adopt
$p.$ It is modified from the actual payoff by not giving her the prize of 1/3
when there is a tie. So it will be enough to show that Player 1 can always
find a $q$ (for any $p$ adopted by the others) with $M\left(  q,p\right)
\geq1/3$ when a tie is possible and consequently her actual payoff will
strictly exceed $1/3.$ So no triple $\left(  p,p,p\right)  $ can constitute an equilibrium.

\begin{lemma}
\label{p<1/2}Suppose two players use a common strategy $p<1/2.$ Then by always
staying at his original node (laziness $q=1$) the remaining player $\left(
1\right)  $ can get a payoff above $1/3.$
\end{lemma}

\begin{proof}
It suffices to show that his payoff for general $p$ is given by the expression
$\left(  1-p\right)  ^{2}/\left(  1-p^{2}\right)  $, which is greater than
$1/3$ for $p<1/2.$ To show this, observe that remaining player wins (payoff 1)
unless exactly one of the remaining players moves before both of them move.
Let $O$ be the event exactly one moves and $B$ be the event both move, $N$ be
the event none moves. The winning sequences for are $B,NB,NNB,.. $ . Since $B$
has probability $\left(  1-p\right)  ^{2}$ and $N$ has probability $p^{2},$
these events have total probability
\begin{align*}
&  \left(  1-p\right)  ^{2}+\left(  1-p\right)  ^{2}p^{2}+\left(  1-p\right)
^{2}\left(  p^{2}\right)  ^{2}+\dots\\
&  =\left(  1-p\right)  ^{2}\left(  1+p^{2}+\left(  p^{2}\right)  ^{2}%
+\dots\right) \\
&  =\left(  1-p\right)  ^{2}/\left(  1-p^{2}\right)  .
\end{align*}
Since this has derivative $-2/\left(  p+1\right)  ^{2}$ it is decreasing and
it's value at $p=1/2$ is $1/3.$
\end{proof}

\begin{lemma}
\label{p>1/2}Suppose two players use a common strategy $p>1/2.$ Then, when
foregoing his payoff of 1/3 in a tie, the remaining player can still obtain a
payoff exceeding $1/3$ by always moving (random walk), $q=0.$ When $p=1/2,$
the payoff is exactly 1/3.
\end{lemma}

\begin{proof}
Let $A\left(  p\right)  $ and $B\left(  p\right)  $ denote the payoff to the
"remaining player" who chooses $q=0$ (always moves) when the others use $p,$
starting respectively with all agents at location 1 (or 3) and all agents at
the middle location $2,$ assuming this player does not accept the payment of
$1/3$ in case of a tie. This last assumption simplifies the algebra. From
position $A$ (all at 1), the remaining player must go to location 2, so there
are three possible subsequent states: all go to middle location 2 (payoff
$B$), the other players stay at location $1$ (payoff 1) or if he alone stays
at location 1. Other outcomes lead to payoff $0$ and can be ignored. This
gives the formula $A=A\left(  q,p\right)  $ in terms of $B=B\left(
q,p\right)  $.
\[
A=\left(  1-p\right)  ^{2}B+p^{2}1
\]

Similarly, if the players all start at the middle location 2, the remaining
player moves to an end (call this end 1). Now there are three subsequent
states: both others stay in the middle (payoff 1), both of the others go to
the same end as the remaining player (payoff $A),$ both of the other players
go to the other end (payoff 1). The other states are either have payoff $0$ or
have payoff 1/3, which we are reducing to $0$ in this calculation. So we have
$B=B\left(  p\right)  $ given by%
\begin{align*}
B  &  =p^{2}1+\left(  1/4\right)  \left(  1-p\right)  ^{2}A+\left(
1/4\right)  \left(  1-p\right)  ^{2}1\\
A  &  =\left(  1-p\right)  ^{2}B+p^{2}1
\end{align*}
We are only interested in the solution $A,$ starting from an end, which is%
\[
A=\frac{-4p+14p^{2}-12p^{3}+5p^{4}+1}{4p-6p^{2}+4p^{3}-p^{4}+3},\text{ which
we want to show is }>\frac{1}{3}\text{.}%
\]
We calculate
\begin{align*}
A-\frac{1}{3}  &  =\frac{8p}{3}\frac{2p^{3}-5p^{2}+6p-2}{-p^{4}+4p^{3}%
-6p^{2}+4p+3}\\
A-\frac{1}{3}  &  =\frac{8p}{3}\frac{\left(  2p-1\right)  \left(
-2p+p^{2}+2\right)  }{\left(  1+2p-p^{2}\right)  \left(  -2p+p^{2}+3\right)  }%
\end{align*}
The denominator is positive for all $p\geq1/2$ and the numerator is positive
for $p>1/2$ and equal to 0 for $p=1/2.$
\end{proof}

\begin{theorem}
There is no symmetric Nash equilibrium for the game $G_{1}\left(  3\right)  .$
\end{theorem}

\begin{proof}
Lemma \ref{p<1/2} shows that $p<1/2$ cannot form a symmetric equilibrium and
Lemma \ref{p>1/2} shows that $p>1/2$ cannot form a symmetric equilibrium.
Consider that Players 2 and 3 adopt $p=1/2$ and Player 1 adopts $q=0.$
According to Lemma \ref{p>1/2} , Player 1 gets a modified payoff of 1/3
(without getting a prize when there is a tie). However a tie has positive
probability. It occurs when 2 and 3 move in the first period and then one
stays in the middle and the other moves to the end node not occupied by 1. So
the payoff (unmodified) to Player 1 in this case exceeds $1/3,$ her payoff
when all three adopt $p=1/2.$
\end{proof}

Of course in this analysis, the laziness strategies should be thought of as
pure strategies. If the players use mixed strategies which are distributions
of $p$'s, there might be a symmetric equilibrium.

\section{Simulation of Social Distancing on the Line and Grid Graphs}

For larger problems with respect to $m$ and $n,$ we determine expected time to
reach social distancing with $D=2$ by simple Monte Carlo simulation methods.
We place the $m$ agents in some specified initial locations on the network.
Then we have them move independently according to LRW's with the same $p$
value. After each step, we find the minimum pairwise distance $d$ between
agents in the current state. If $d\geq D$ ($=2$ for the examples here), we
stop and record the time $T.$ We carry out $5,000$ trials and record the mean.
Contrary to our earlier results, we find for the line and the two dimensional
grid that it is optimal for the agents to follow (independent) random walks,
$p=0.$ When $n$ is very small, it takes a little longer to reach social distancing.

\subsection{The two dimensional $k\times k$ grid $GR_{k}$}

In practice, social distancing is often to be achieved by individuals in a
planar region. A good network model for this is the two dimensional grid graph
$GR_{k}$ with $n=k^{2}$ nodes in the set $\left\{  (i,j\right\}  :1\leq
i,j\leq k\},$ as shown in Figure 13.

    \begin{figure}[H]
     \centering
     \includegraphics[width=0.95\textwidth]{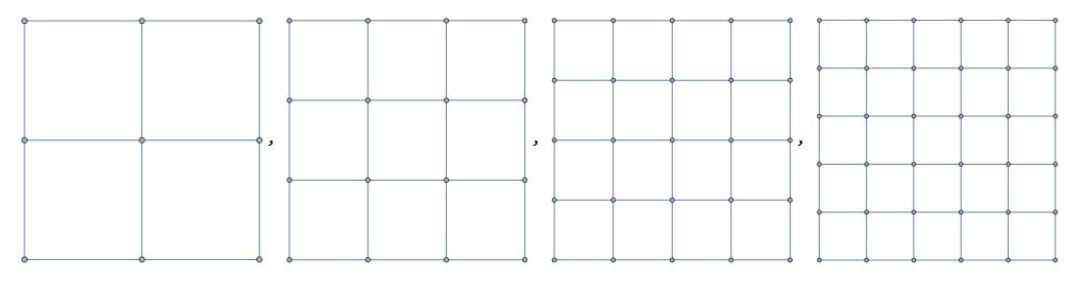}
     \caption{Two dimensional grids $GR_k,k=3,...,6.$.}
     \label{fig:my_label}
    \end{figure}

A natural starting state is the one with all agents at a corner node (say
$\left(  1,1\right)  )$ or at the center (both coordinates $\left\lfloor
k/2\right\rfloor $. Figure 14 illustrates these times for values of $p$ spaced
at distance $0.2.$ Note that for all the four values of $k,$ the mean times to
reach distance $d=2$ are increasing in $p$. The means that the random walk,
$p=0,$ is the best. In terms of grid size $k,$ It takes a bit longer for the
$3\times3$ grid because reflections from the boundary are more common. For
larger values of $k$ the times do not appear to depend much on $k.$
  \begin{figure}[H]
     \centering
     \includegraphics[width=0.8\textwidth]{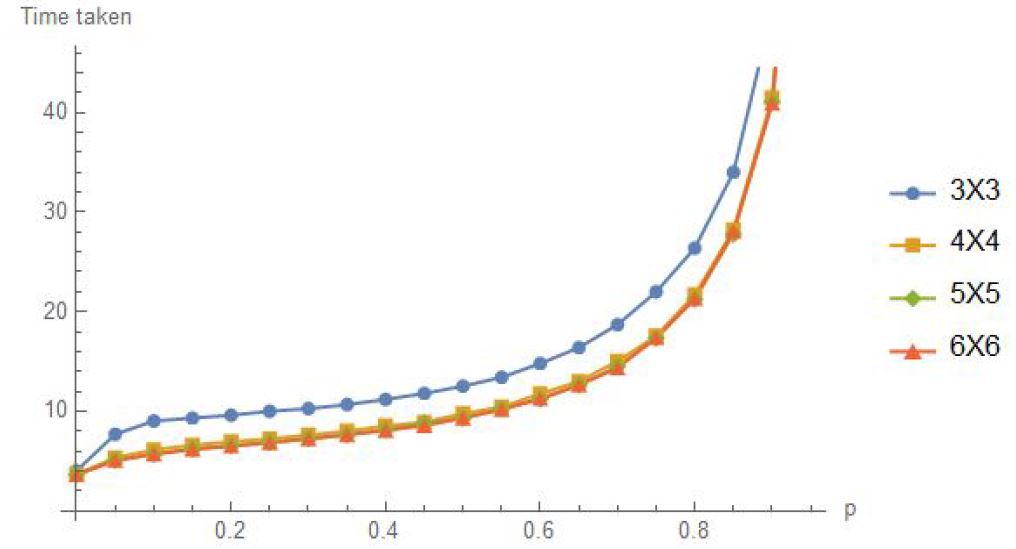}
     \caption{Time to $d=D\equiv2$ on $GR_k,k=3$ to $6,$ from corner start.}
     \label{fig:my_label}
    \end{figure}
If the starting state consists of all agents at the center of the grid then we
have similar result, as seen in Figure 15.%
  \begin{figure}[H]
     \centering
     \includegraphics[width=0.8\textwidth]{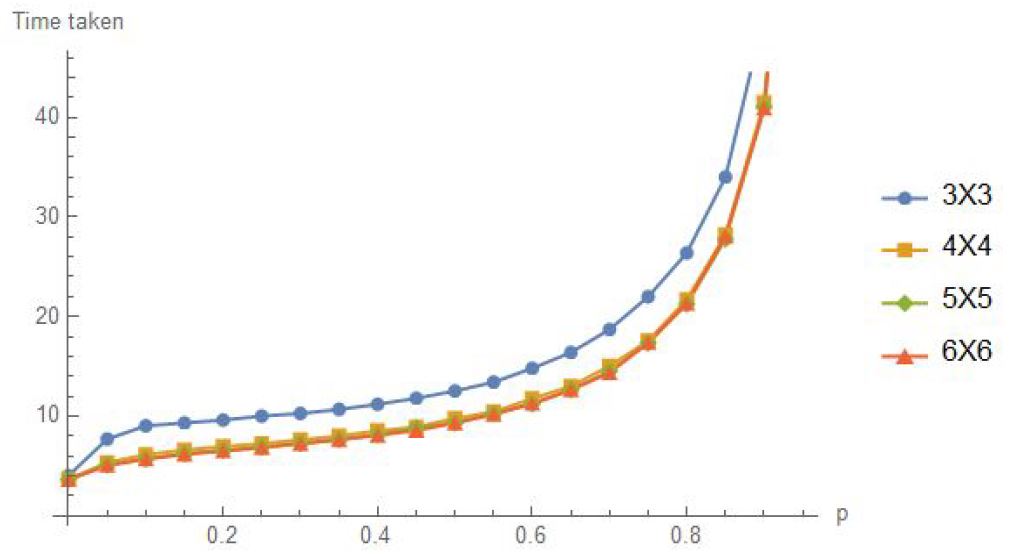}
     \caption{Time to $d=2$ from center start.}
     \label{fig:my_label}
    \end{figure}

\subsection{The line graph $L_{n}.$}

\qquad The graph $L_{n}$ has $n$ nodes arranged in a line and numbered from
the left as $1$ to $n.$ Like the grid graph, a natural starting state is
either all at an end (say node $1$) or all at the center. We find that the
common value of $p$ should be $0,$ that is, the agents should adopt
independent random walks. Figure 16 shows this for a left start and Figure 17
shows this for a center start, at $\left\lfloor n/2\right\rfloor $.
   \begin{figure}[H]
     \centering
     \includegraphics[width=0.8\textwidth]{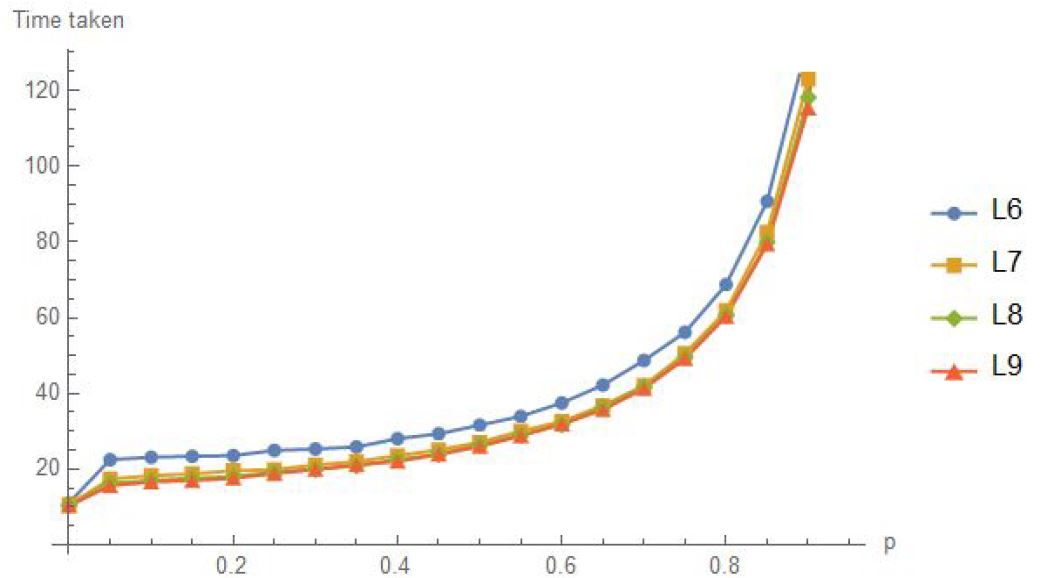}
     \caption{Time to $d=2$ for left node start on $L_n.$}
     \label{fig:my_label}
    \end{figure}

    \begin{figure}[H]
     \centering
     \includegraphics[width=0.85\textwidth]{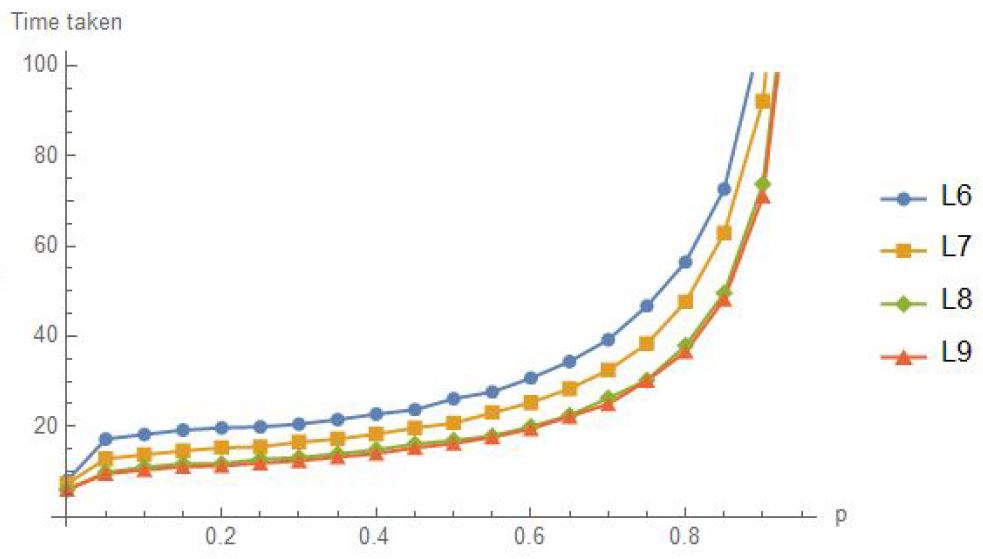}
     \caption{Time to $d=2$ for a central start.}
     \label{fig:my_label}
    \end{figure}
\section{Conclusions}

This article introduced the Social Distancing Problem on a connected graph,
where agents have a common goal to have all their pairwise distances be at
least a given number $D.$ While different motions and information could be
given to the agents for this problem, we give them only local knowledge of the
graph and no knowledge of locations of other agents. So they know only the
degree of their current node and lack memory. These assumptions limit the
motions of the agents to Lazy Random Walks. We showed how to optimize their
common laziness value $p$ to achieve social distancing in the least expected
number of steps. We considered various graphs and both exact and simulated
methods. In some cases the optimal motion was a random walk ($p=0$) or a
loop-random walk (choosing their current node with the same probability as
each adjacent one). We also considered variations where agents know the
current population $k$ of their node and can choose laziness $p_{k}$
accordingly. While mostly we consider the common-interest team version of the
problem, we also studied cases where agents had individual selfish motives -
we showed that in some cases no symmetric equilibrium exists.

We expect this area of research to be enlarged to other assumptions:

\begin{itemize}
\item Agents know locations of some or all of the other agents.

\item Agents have some memory.

\item Agents know the whole graph.

\item Agents can gain `territoriality over a node'.
\end{itemize}

It turns out that our model of mobile agents on a graph is also useful for
some other problems (goals). One goal is multi rendezvous, or gathering, where
the common goal is for all agents to occupy a common node. This extended
earlier results limited to two agents. Another problem is the search game
where agents come in two types, searchers and hiders, with obvious associated
goals. Here our methods extend known results to multiple searchers. Many other
results in search games could usefully be extended in a similar way.

In this first paper on social distancing, we have restricted ourselves to
considering only some simple classes of graph and small sizes. It is to be
hoped that further research in this area will find new and stronger methods
able to study general graphs.

\end{document}